\documentclass[a4paper,11pt,reqno]{amsart} 
\usepackage{amsmath}
\usepackage{amsfonts}
\usepackage[T1]{fontenc}
\usepackage[utf8]{inputenc}
\usepackage[mathscr]{euscript}
\usepackage{verbatim}
\usepackage{amssymb,latexsym}
\usepackage{amsthm}
\usepackage{eufrak}
\usepackage{color}
\usepackage{graphicx}

\newcommand{\Ab}{\mathbf{A}}
\usepackage[lmargin=2.5 cm,rmargin=2.5 cm,tmargin=3.5cm,bmargin=2.5cm,paper=a4paper]{geometry}
\DeclareMathOperator{\curl}{curl}\DeclareMathOperator{\Div}{div}

\newtheorem{thm}{Theorem}[section]

\newtheorem{lem}[thm]{Lemma}

\newtheorem{theorem}[thm]{Theorem}
\newtheorem{assumption}[thm]{Assumption}

\newtheorem{lemma}[thm]{Lemma}
\newtheorem{proposition}[thm]{Proposition}
\newtheorem{corollary}[thm]{Corollary}

\theoremstyle{remark}
\newtheorem{rem}[thm]{Remark}

\newcommand{\nb}{\nabla}
\newcommand{\R}{\mathbb{R}}
\newcommand{\Fb}{\mathbf{F}}

\newcommand{\N}{\mathbb{N}}
\newcommand{\C}{\mathbb{C}}

\newcommand{\Hd}{H_{{\rm div}}^1(\Omega)}
\newcommand{\Om}{\Omega}
\newcommand{\kp}{\kappa}
\newcommand{\kn}{\nabla-i\kappa H{\bf A}}
\newcommand{\Es}{{\rm E}_{\rm g.st}(\kappa,  H)}
\newcommand{\GL}{\mathcal E_{\kappa,H}}
\newcommand{\Gb}{G_{b,Q_R}^\sigma}

\def\sig#1{\vbox{\hsize=5.5cm
\kern2cm\hrule\kern1ex
\hbox to \hsize{\strut\hfil #1 \hfil}}}

\newcommand\signatures[4]{%
\vspace{3cm}
\hbox to \hsize{\hfil #1, \today\hfil}
\vspace{3cm}
\hbox to \hsize{\quad#2\hfil\hfil #3\quad}
\vspace{3cm}
\hbox to \hsize{\hfil#4\hfil}}
\numberwithin{equation}{section}

\title[Ginzburg-Landau Functional]{The influence of magnetic steps on bulk superconductivity}
\author[W. Assaad]{Wafaa Assaad}
\address{Lebanese International  University, Department of Mathematics, Saida Campus, Lebanon}
\email{Wafaa\_assaad@hotmail.com}
\author[A. Kachmar]{Ayman Kachmar}
\address{Lebanese University, Department of Mathematics, Hadat, Lebanon}
\email{ayman.kashmar@gmail.com}
\date{\today}
\begin{document}
\maketitle
\begin{abstract}
We study the distribution of bulk superconductivity in presence of an applied  magnetic field, supposed to be a step function, modeled by  the Ginzburg-Landau theory. Our results are valid for the minimizers of the two-dimensional Ginzburg-Landau functional with a large Ginzburg-Landau parameter and with an applied magnetic field of intensity comparable with the Ginzburg-Landau parameter. 
\end{abstract}

\section{Introduction and Main results}\label{sec:int}

\subsection*{Motivation}
The Ginzburg-Landau functional successfully models the response of a (Type~II) superconducting sample to an applied magnetic field. We focus on samples that  occupy a long cylindrical domain and subjected to a magnetic field with direction parallel to the axis of the cylinder. This situation has been analyzed in many papers, see for instance the two monographs \cite{Fr_Hl_10, SS-b}. However, in the literature, the focus was on uniform applied magnetic fields. Recently, attention has been shifted to non-uniform {\it smooth} magnetic fields in \cite{att1, att2, HeKa, PK}. Such magnetic fields may arise in the study of superconducting surfaces \cite{CL} or superconductors with applied electric currents \cite{AHP}. 

In this paper, we consider the situation when the applied magnetic field is a step function. Such fields might occur in many situations (cf. \cite{HPRS}). In particular
\begin{itemize}
\item If a sample is separated into two parts, one can apply on one part a uniform magnetic field from above, and on the other part, a uniform magnetic field from below (see Figure~2).
\item If a sample is not homogeneous, one can have a variable magnetic permeability. 
This leads to a magnetic step function (cf. \cite{chapman}).   
\end{itemize}
 
\subsection*{The functional}

Let $\Omega\subset\R^2$ be an open, bounded and simply connected set and $B_0:\Omega\to[-1,1]$ be a measurable function. The Ginzburg-Landau functional  in $\Omega$ is
\begin{equation}\label{eq:GL}
\GL (\psi,\Ab)= \int_\Om \Big( \big|(\nb-i\kp H {\bf
A})\psi\big|^2-\kp^2|\psi|^2+\frac{\kappa^2}{2}|\psi|^4 \Big)\,dx
 +\kp^2H^2\int_\Om\big|{\rm curl}{\bf A}-B_0\big|^2\,dx.
\end{equation}
Here, $\kp>0$ is the Ginzburg-Landau parameter, a characteristic of the superconducting material, $H > 0$ is the intensity of the applied magnetic field,  $\psi\in H^1(\Om, \C)$ and  ${\bf A}=(A_1,A_2) \in H^1(\Om, \R^2)$. In physics, the domain $\Omega$ is the cross section of the sample, the function $B_0$ is the applied magnetic field, the function $\psi$ is  the order parameter and the vector field $\Ab$ is  the magnetic potential. The configuration $(\psi,\Ab)$ is interpreted as follows, $|\psi|^2$ measures the density of the superconducting electron pairs and $\curl\Ab=\partial_{x_1}A_2-\partial_{x_2}A_1$ measures the induced magnetic field in the sample.

In this paper, we work under the following assumption on the function $B_0$  (these are illustrated in Figure~\ref{fig1}):

\begin{assumption}\label{assump}~
\begin{enumerate}
\item $a\in[-1,1)\setminus\{0\}$ is a given constant\,;
\item $\Omega_1\subset\Omega$ and $\Omega_2\subset\Omega$ are two disjoint open sets\,;
\item $B_0={\mathbf 1}_{\Omega_1}+a{\mathbf 1}_{\Omega_2}$\,;
\item $\Omega_1$ and $\Omega_2$ have a finite number of connected components\,;
\item $\partial\Omega_1$ and $\partial\Omega_2$ are piecewise smooth with (possibly) a finite number of corners\,;
\item $\Gamma=\partial\Omega_1\cap\partial\Omega_2$ is the union of a finite number of smooth curves\,;
\item $\Omega=\Omega_1\cup\Omega_2\cup\Gamma$ and $\partial\Omega$  is of class $C^4$\,;
\item $\Gamma\cap\partial\Omega$ is either empty or finite\,;
\item If $\Gamma\cap\partial\Omega$, then $\Gamma$ intersects $\partial\Omega$ transversely, i.e. $\nu_{\partial \Omega} \times \nu_\Gamma \neq 0$, where $\nu_{\partial \Omega}$ and  $\nu_\Gamma$ are respectively the unit  normal vectors of $\partial \Omega$ and $\Gamma$.
\end{enumerate}
\end{assumption}
\begin{figure}\label{fig1}
\centering
\includegraphics[scale=0.7]
{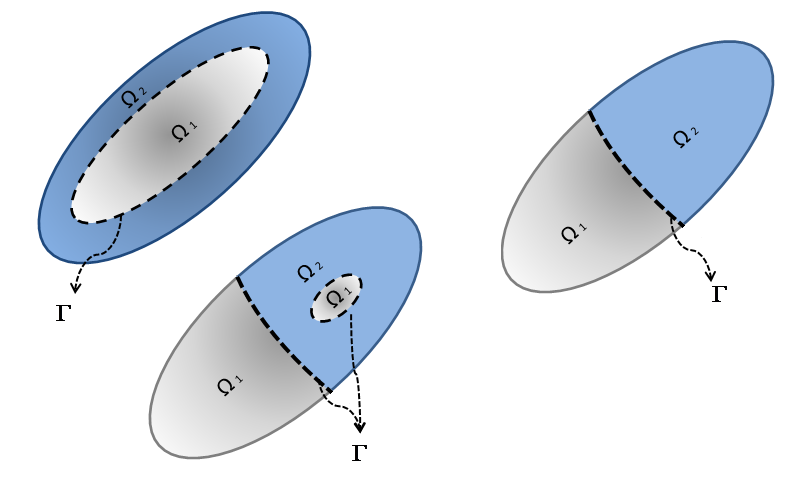} 
\caption{Schematic representations of the set $\Om$.} 
\end{figure}
\begin{figure}\label{fig2}
\centering
\includegraphics[scale=0.7]{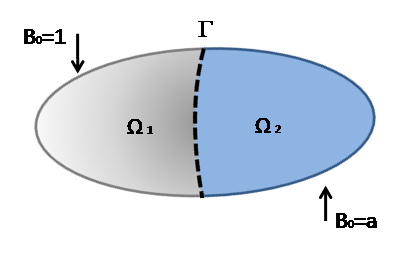} 
\caption{Schematic representations of the set $\Om$ subjected to a step magnetic field $B_0$. }
\end{figure}
We introduce the ground state energy
\begin{equation} \label{eq:gr_st}
\Es=\inf\{\GL(\psi,\Ab)~:~(\psi,{\bf A}) \in H^1(\Om;\C)\times\Hd\}
\end{equation}
where
\begin{equation}\label{eq:Hd}
\Hd=
\left\{
 {\bf A} \in H^1(\Om,\R^2)~:~ {\rm div}{\bf A}=0 \ {\rm in}\ \Om,\ {\bf A}\cdot\nu_{\partial\Omega}=0\ {\rm on}\ \partial \Om
\right\}\,.
\end{equation}
The functional $\GL$ is invariant under the gauge transformations $(\psi,\Ab)\mapsto (e^{i\varphi\kappa H}\psi,{\bf A}+ \nb \varphi)$. This gauge invariance yields (cf. \cite{Fr_Hl_10})
\begin{equation*}
\Es=\inf\{\GL(\psi,\Ab)~:~(\psi,{\bf A}) \in H^1(\Om;\C)\times H^1(\Om; \R^2)\}\,.
\end{equation*}

Critical points $(\psi, {\bf A}) \in H^1(\Om, \C)\times\Hd $ of $\GL$ are weak solutions of the following G-L equations:
\begin{equation}\label{eq:Euler}
\left\{
\begin{array}{ll}
\big(\kn\big)^2\psi=\kp^2(|\psi|^2-1)\psi &{\rm in}\ \Om\,,\\
-\nb^{\perp}  \big(\curl\Ab-B_0\big)= \frac{1}{\kp H}{\rm Im}\big(\overline{\psi}(\nb-i\kp H {\bf A})\psi\big) & {\rm in}\ \Om\,,\\
\nu\cdot(\kn)=0 & {\rm on}\ \partial \Om \,,\\
{\rm curl}{\bf A}=B_0 & {\rm on}\ \partial \Om \,.
\end{array}
\right.
\end{equation}
Here  $\nb^\perp =(\partial_{x_2},-\partial_{x_1})$ is the Hodge gradient.

\subsection*{Energy asymptotics and applications}

The statement of our main results involves a continuous function $g:[0,\infty)\to[-\frac12,0]$ constructed in \cite{SS02, FK-cpde}. The function $g$ is increasing and satisfies $g(0)=-\frac12$, $g(b)=0$ for all $b\geq1$, and $-\frac12<g(b)<0$ for all $b\in(0,1)$.

\begin{theorem}\label{thm:En_estim}{\bf[Ground state energy asymptotics]}~

Let $\tau\in(\frac32,2)$ and $0<c_1<c_2$ be constants. Under Assumption~\ref{assump}, there exist constants $C>0$ and $\kp_0>0$  such that if
\begin{equation}\label{eq:order}
\kp\geq \kp_0 \quad {\rm and}\quad c_1\leq \frac{H}{\kp}\leq c_2\,,
\end{equation}
then the ground state energy in \eqref{eq:gr_st} satisfies
\begin{equation}\label{eq:gr_estim}
-C\kappa^\tau\leq \Es-\kp^2\int_\Om g\left (\frac{H}{\kp}|B_0(x)| \right)dx \leq C\kp^{3/2}\,.
\end{equation}
\end{theorem}
\begin{corollary}\label{cor:psi_estim}{\bf[$L^4$-norm asymptotics of the order parameter]}~

Under the assumptions of Theorem~\ref{thm:En_estim}, there exist
constants $\kappa_0>0$ and $C>0$ such that, if  \eqref{eq:order}
holds, then:
\begin{enumerate}
\item For every critical point $(\psi,\Ab)$ of \eqref{eq:GL},
$$\int_\Om |\psi|^4\,dx+2\int_\Om g\left(\frac{H}{\kappa}|B_0(x)|\right)\, dx\leq C\kappa^{\tau-2}\,.$$
\item For every minimizer $(\psi,{\bf A})$  of~\eqref{eq:GL},
$$\left| \int_\Om |\psi|^4\,dx+2\int_\Om g\left(\frac{H}{\kappa}|B_0(x)|\right)\, dx \right|\leq C \kappa^{\tau-2}\,,$$
and
$$\kappa^2H^2\int_\Omega|\curl\Ab-B_0|^2\,dx\leq C\kappa^\tau\,.$$
\end{enumerate}
\end{corollary}
\begin{theorem}\label{thm:psi-estim}{\bf[Distribution of the $L^4$-norm of the order parameter]}~

Suppose that the assumptions of Theorem~\ref{thm:En_estim} are
satisfied. Let $D\subset\Omega$ be an open set with a smooth
boundary. There exist  $\kappa_0>0$   and a function ${\rm err}:(\kappa_0,\infty)\to\R_+$ such that $\displaystyle\lim_{\kappa\to\infty}{\rm err}(\kappa)=0$ and, if
\begin{itemize}
\item \eqref{eq:order} holds\,;
\item $(\psi,\Ab)$ is a minimizer of
\eqref{eq:GL}\,;
\end{itemize}
 then
$$\left| \int_D |\psi|^4\,dx+2\int_D g\left(\frac{H}{\kappa}|B_0(x)|\right)\, dx \right|\leq {\rm err}(\kappa)\,.$$
\end{theorem}

\subsection*{Discussion of Theorem~\ref{thm:psi-estim}}

The result in Theorem~\ref{thm:psi-estim} displays the strength of the superconductivity in the bulk of $\Omega$. We will use the following notation. Let $\omega\subset\R^2$ be an open set, $(f_\kappa)$ be a sequence in $L^\infty(\omega)$, $\alpha\in\C$ and $dx$ be the Lebesgue measure in $\R^2$. By writing
$$f_\kappa dx\rightharpoonup \alpha dx\quad{\rm in~}\omega$$
we mean, for every ball $B\subset \omega$,
$$\int_B f_\kappa dx\to \alpha |B|\,.$$
Now we return back to the result in Theorem~\ref{thm:psi-estim}. Suppose that $H=b\kappa$ and $b\in(0,\infty)$ is a constant. Let us start by examining the case where $-1<a<1$. We observe that:
\begin{enumerate}
\item If $0<b<1$, then
$$|\psi|^4\,dx\rightharpoonup -2g(b)\,dx~{\rm in~}\Omega_1\quad{\rm and}\quad |\psi|^4\,dx\rightharpoonup -2g(b|a|)\,dx~{\rm in~}\Omega_2\,. $$
This means that the bulk of $\Omega$ carries superconductivity everywhere, but since $0<|a|<1$, $0<-g(b)<-g(b|a|)$ and the strength of superconductivity in $\Omega_1$ is smaller than that in $\Omega_2$.
\item If $1\leq b<\frac1{|a|}$, then 
$$|\psi|^4\,dx\rightharpoonup 0~{\rm in~}\Omega_1\quad{\rm and}\quad |\psi|^4\,dx\rightharpoonup -2g(b|a|)\,dx~{\rm in~}\Omega_1\,, $$
with $g(b|a|)<0$. In this regime, superconductivity disappears in the bulk of $\Omega_1$ but persists in the bulk of $\Omega_2$. Theorem~\ref{thm:exp-dec} below will sharpen this point by establishing that $|\psi|$ is exponentially small in the bulk of $\Omega_1$ (see Figure~3). However, in light of the analysis in the book of Fournais-Helffer \cite{Fr_Hl_10}, the boundary of $\Omega_1$ may carry superconductivity. This point deserves a detailed analysis.
\item If $b> \frac1{|a|}$, then superconductivity disappears in the bulk of $\Omega_1$ and $\Omega_2$ (see Figure~4). However, one might find an interesting behavior near the critical value $b\sim \frac1{|a|}$. In the spirit of the   analysis in \cite{FK-adm},  one expects to find  superconductivity in the bulk of $\Omega_2$, but  with a weak strength. This superconductivity is   evenly  distributed   and  decays as $b$ gradually increases past  the value $\frac1{|a|}$.
\item {\bf[Break down of superconductivity (\cite{GP})]} If $b\gg \frac1{|a|}$, one expects that $\psi=0$ and superconductivity is lost in the sample. To this end, the spectral analysis  in \cite{HPRS} must be useful. In the spirit of the book \cite{Fr_Hl_10}, this regime is related to the analysis of the third critical field(s) where the transition to the purely normal state occurs.
\end{enumerate}

\begin{figure}
\centering
\begin{minipage}{0.47\textwidth}
\centering
\includegraphics[scale=0.4]{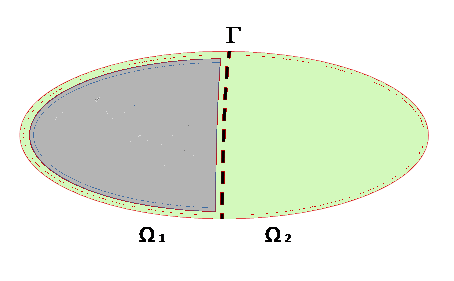}\label{Fig:3}
\caption{Schematic representation of $\Omega$ in the regime $\frac H \kp=b$ where $1 < b< \frac 1 {|a|}$.
	The dark  region is in a normal state, the bright region may carry superconductivity.}
\end{minipage}\hfill
\begin{minipage}{0.47\textwidth}
\centering
\includegraphics[scale=0.4]{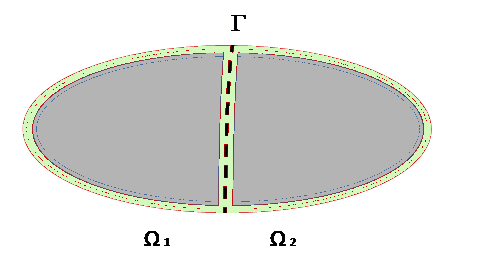}\label{Fig:4}
\caption{Schematic representation of $\Omega$ in the regime $\frac H \kp=b$ where $ b > \frac 1 {|a|}$.
	The dark region is in a normal state, the bright region may carry superconductivity.}
\end{minipage}
\end{figure}

The interesting case $a=-1$ is reminiscent of the situation of  a {\it smooth} and sign-changing   magnetic field analyzed  in the paper by Helffer-Kachmar \cite{HeKa}. Note that Theorem~\ref{thm:psi-estim} yields that superconductivity is evenly distributed in $\Omega_1$ and $\Omega_2$ as long as $0<b<1$. In the critical regime $b\sim 1$, one might find that superconductivity is distributed along the curve $\Gamma$ that separates $\Omega_1$ and $\Omega_2$, in the same spirit of the paper \cite{HeKa}. This behavior is illustrated in Figure~5.
\begin{figure}
\centering
\includegraphics[scale=0.5]{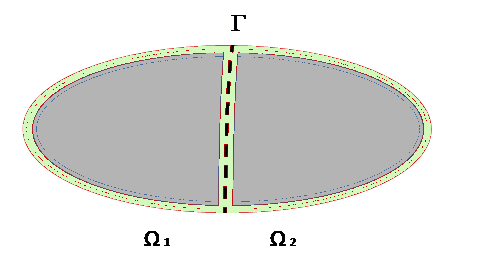} \label{fig5}
\caption{Schematic representation of the set $\Om$ in the regime $\frac H \kp = b$ where $a=-1$ and $b>1$. The dark region is in a normal state, the bright region may carry superconductivity. }
\end{figure} 

\subsection*{Exponential decay in regions with larger magnetic intensity}

Our last result establishes a regime for the strength of the magnetic field where the order parameter is exponentially small in the bulk of $\Omega_1$. The relevance of this theorem is that together with Theorem~\ref{thm:psi-estim}, display a regime of the intensity of the applied magnetic field  such that $|\psi|^2$ is exponentially small in the bulk of $\Omega_1$ while it is of the order   
$O(1)$ in $\Omega_2$.


\begin{thm}\label{thm:exp-dec}{\bf[Exponential decay of the order parameter]}~

Let $\lambda,\varepsilon,c_2>0$ be   constants such that $0<\varepsilon<\sqrt{\lambda}$ and $1+\lambda< c_2$. 
There exist constants $C,\kappa_0>0$ such that, if 
$$\kappa\geq\kappa_0\,,\quad (1+\lambda)\kappa \leq H\leq c_2\kappa\,,\quad (\psi,\Ab)_{\kappa,H}~{\rm is~a~minimizer~of~}\eqref{eq:GL}\,,$$ 
then
\begin{multline*}
\int_{\Omega_1\cap\{{\rm dist}(x,\partial\Omega_1)\geq \frac{1}{\sqrt{\kappa H}}\}}
\Big(|\psi|^2+{\frac{1}{\kp H}}|(\nabla-i\kappa H\Ab)\psi|^2\Big)\,\exp\Big(2\varepsilon\sqrt{\kappa H}\,{\rm dist}(x,\partial\Omega_1)\Big)dx\\
\leq
C \int_{\Omega_1\cap\{{\rm dist}(x,\partial\Omega_1)\leq \frac{1}{\sqrt{\kappa H}}\}}
|\psi|^2\,dx\,.\end{multline*}
\end{thm}

Unlike similar situations in \cite{BF, FK-jde}, we can not extend the result in  Theorem~\ref{thm:exp-dec} to {\it critical points} of the functional in \eqref{eq:GL}. The technical reason behind this is as follows. A necessary ingredient in the proof given in \cite{BF, FK-jde} is the following estimate of the magnetic energy
$$\|\curl\Ab-B_0\|_{L^2(\Omega)}=o(\kappa^{-1})\quad(\kappa\to\infty)\,.$$
For critical configurations, we have the following estimate from \cite[Lemma~10.3.2]{Fr_Hl_10}
$$\|\curl\Ab-B_0\|_{L^2(\Omega)}\leq C\kappa^{-1}\|\psi\|_{L^2(\Omega)}\|\psi\|_{L^4(\Omega)}\,.$$
To control the $L^2$- and $L^4$- norms of $\psi$, we use Theorem~\ref{thm:psi-estim}. But this will give that $\|\psi\|_{L^4(\Omega)}=o(1)$ only for $H\geq |a|^{-1}\kappa$, the condition necessary to get that $g(H\kappa^{-1})=g(H\kappa^{-1} |a|)=0$. This condition does not cover all the values of $H$ in Theorem~\ref{thm:exp-dec}. As a substitute, we choose to control the magnetic energy by the estimate in Corollary~\ref{cor:psi_estim}, which is valid for {\it minimizing} configurations only.

\subsection*{Notation}~
\begin{itemize}
\item The letter $C$ denotes a positive constant whose value may change from line to line. Unless otherwise stated, the constant $C$ depends on the function $B_0$ and the domain $\Omega$, and independent of $\kappa$, $H$ and the minimizers $(\psi,\Ab)$ of the functional in \eqref{eq:GL}.
\item Given $\ell>0$ and $x=(x_1,x_2) \in \R^2$, we denote by 
$$Q_{\ell}(x)=(-\frac \ell 2+x_1,\frac \ell 2+x_1)\times(-\frac \ell 2+x_2,\frac \ell 2+x_2)$$ the square of side length $\ell$ centered at $x$ .
\item Let $a(\kp)$ and $b(\kp)$ be two positive functions, we write :
\begin{itemize}
\item $a(\kp) \ll b(\kp)$, if $\frac {a(\kp)}{b(\kp)} \rightarrow 0$ as $\kp \rightarrow \infty$.
\item $a(\kp) \approx b(\kp)$, if there exist constants $\kappa_0$, $C_1$ and $C_2$ such that for all $\kp\geq\kp_0,\,C_1 a(\kp)\leq b(\kp) \leq C_2 a(\kp).$
\end{itemize}
\item The quantity $o(1)$ indicates a function of $\kappa$, defined by universal quantities, the domain $\Omega$, given functions, etc and  such that $|o(1)|\ll1$. Any expression $o(1)$ is independent of the minimizer $(\psi,\Ab)$ of \eqref{eq:GL}. Similarly, $O(1)$ indicates a  function of $\kappa$, bounded by a constant independent of the minimizers of \eqref{eq:GL}. 
\item Let $n \in \N$, $p \in \N$.   We use the following Sobolev spaces :
$$W^{n,p}(\Omega):=\left\{f \in L^p(\Omega) \ | \ D^\alpha f \in L^p(\Omega),\ {\rm for\ all}\ |\alpha|\leq n \right\},$$
$$H^n(\Omega):= W^{n,2}(\Omega).$$
\end{itemize}

\subsection*{On the proofs and the organization of the paper}~
The results in this paper can be viewed as generalizations of those in \cite{SS02} already proved for the case $B_0=1$. Theorem~\ref{thm:exp-dec} is reminiscent of the exponential bounds in \cite{BF}. However, the proofs in this paper  are simpler than those in \cite{SS02} and contain new ingredients that we summarize below:
\begin{itemize}
\item We took advantage of all the available information regarding the limiting function $g(\cdot)$ proved in \cite{FK-cpde} and \cite{att2}\,;
\item We did not used the {\it a priori} elliptic estimates, e.g.  the $L^\infty$-bound  $\|(\nabla-i\kappa H\Ab)\psi\|_\infty\leq C\kappa$. However, we used the simple energy bound $\|(\nabla-i\kappa H\Ab)\psi\|_2\leq C\kappa$ together with the regularity of the $\curl$-${\rm div}$ system (cf. Theorem~\ref{thm:priori}). This method is already used for the three dimensional problem in \cite{K-JFA}\,;
\item To prove Theorem~\ref{thm:exp-dec}, we did not established {\it weak} decay estimates as done in \cite{BF}.   
\end{itemize} 

The paper is divided into seven sections and two appendices. The first section is this introduction. Section~\ref{sec:lim-en} collects the needed properties of  the limiting energy $g(\cdot)$. Section~\ref{sec:ub} establishes an upper bound of the ground state energy. Section~\ref{sec:a-p-est} proves the necessary estimates on the critical points of the functional in \eqref{eq:GL}. These estimates are used in Section~\ref{sec:lb} to establish a lower bound of the ground state energy. In Section~\ref{sec:proof}, we finish the proof of Theorem~\ref{thm:En_estim}, Corollary~\ref{cor:psi_estim} and Theorem~\ref{thm:psi-estim}. Section~\ref{sec:exp-dec} is devoted to the proof of Theorem~\ref{thm:exp-dec}. Finally, the appendices \ref{sec:gauge} and \ref{sec:reg} collect standard results that are used throughout  the paper.  
\section{The limiting energies}\label{sec:lim-en}
Let $R>0$ and $Q_R=(-R/2,R/2)\times(-R/2,R/2)$. We define the following Ginzburg-Landau energy with the constant magnetic field on $H^1(Q_R)$ by

\begin{equation}\label{eq:Gb}
\Gb(u)=\int_{Q_R} \left(b|(\nb-i\sigma {\bf A}_0)u|^2-|u|^2+\frac 12 |u|^4\right)\,dx\,.
\end{equation}

Here $b\geq 0$, $\sigma \in \{ -1,+1\}$ and ${\bf A}_0$ is the canonical magnetic potential
\begin{equation}\label{canon}
{\bf A}_0 = \frac 12 (-x_2,x_1) \qquad \forall x = (x_1,x_2) \in \R ^2,
\end{equation}
which satisfies $\curl\Ab_0=1$.

We introduce the two ground state energies
\begin{equation}\label{m0-mN}
\begin{aligned}
&m_0(b,R,\sigma)=\displaystyle\inf_{u \in H_0^1(Q_R)}
G_{b,Q_R}^\sigma(u)\,,\\
&m(b,R,\sigma)=\displaystyle\inf_{u \in H^1(Q_R)}
G_{b,Q_R}^\sigma(u).
\end{aligned}
\end{equation}

Notice that $G_{b,Q_R}^{+1}(u) =G_{b,Q_R}^{-1}(\overline{u})$. As an
immediate consequence, we observe that
\begin{equation}\label{eq:inf}
\displaystyle\inf_{u \in \mathcal V} G_{b,Q_R}^{+1}(u) = \displaystyle\inf_{u \in \mathcal V} G_{b,Q_R}^{-1}(u)\quad{\rm where}\quad
\mathcal V\in\{H_0^1(Q_R),H^1(Q_R)\},
\end{equation}
and the values of $m_0(b,R,\sigma)$ and $m(b,R,\sigma)$ are
independent of $\sigma\in\{-1,1\}$.  In the rest of the paper, we
will denote these  two values by $m_0(b,R)$ and $m(b,R)$
respectively, hence
\begin{equation} \label{eq:m}
m_0(b,R,\sigma)=m_0(b,R)\quad{\rm and}\quad m(b,R,\sigma)=m(b,R)\quad(\sigma\in\{-1,1\})\,.
\end{equation}

We cite the following result from \cite{att2} (also see \cite{
FK-cpde, SS02}).

\begin{thm}\label{thm:g_prop}~
\begin{enumerate}
\item For all $b\geq1$ and $R>0$, we have $m_0(b,R)=0$.
\item For all $b \in [0,\infty)$, there exists a constant $g(b)\leq 0$ such that
$$g(b)= \lim_{R\rightarrow \infty}\frac{m_0(b,R)}{R^2}=\lim_{R\rightarrow\infty} \frac{m(b,R)}{R^2} \quad {\rm and} \quad g(0)=-\frac 12.$$
\item The function $[0,\infty)\ni b \mapsto g(b)$ is continuous, non-decreasing, concave and its range is the interval $[-\frac 12,0]$.
\item There exists a constant $\gamma \in (0,\frac 12)$ such that
$$\forall b \in [0,1],\quad \gamma(b-1)^2\leq|g(b)|\leq \frac 12 (b-1)^2.$$
\item There exist constants $C$ and $R_0$ such that, for all $R\geq
R_0$ and $b\in[0,1]$,
$$
g(b)\leq
\frac{m_0(b,R)}{R^2}\leq g(b)+\frac{C}{R}\\
\quad{\rm and}\quad g(b)-\frac{C}R\leq \frac{m(b,R)}{R^2}\leq g(b)+\frac{C}{R}\,.
$$
\end{enumerate}
\end{thm}

\section{Energy Upper Bound}\label{sec:ub}

The aim of this section is to prove :

\begin{proposition}\label{prop:up}
Under the assumption of Theorem~\ref{thm:En_estim}, there exist
positive constants $C$ and $\kp_0$ such that if \eqref{eq:order}
holds, then the ground state energy $\Es$ in \eqref{eq:gr_st}
satisfies
$$\Es \leq \kp^2\int_\Om g\Big (\frac{H}{\kp}|B_0(x)|\Big )dx + C\kp^{3/2}.$$
\end{proposition}

Before writing the proof of Proposition~\ref{prop:up}, we introduce
some notation.  If $D\subset\Omega$ is an open set, we introduce the
local energy of the configuration $(\psi,{\bf A}) \in
H^1(\Om;\C)\times\Hd$ in the domain $D\subset\Om$ as follows
\begin{equation}\label{eq:local0}
\begin{aligned}
&\mathcal E_0(\psi,{\bf A};D) =\int_D\big( |(\nb-i \kp H {\bf A})\psi|^2-\kp^2|\psi|^2 +\frac 12 \kp^2|\psi|^4 \big)\,dx\,,\\
&\mathcal E(\psi,{\bf A};D)=\mathcal E_0(\psi,{\bf A};D) +(\kp H)^2 \int_\Om|{\rm curl(}{\bf A-F)|}^2\,dx.
\end{aligned}
\end{equation}

In Lemma~\ref{A_1}, we constructed a vector field  ${\bf F}$
satisfying
\begin{equation}\label{eq:F}
{\bf F} \in \Hd\quad{\rm and}\quad \curl{\bf F}= B_0\quad{\rm in}~\Om\,.
\end{equation}

\begin{proof}[Proof of Proposition~\ref{prop:up}]~

{\bf Step 1.} (Introducing a lattice of squares)

We introduce the small parameter
\begin{equation}\label{eq:ub-ell}
\ell=\kappa^{-1/2}\,.
\end{equation}
Consider the lattice $L_\ell:=\ell \mathbb Z \times \ell\mathbb Z$. Let
\begin{equation}\label{eq:I-ell}
\mathcal I_\ell^1= \{z \in L_\ell~:~\overline{Q_\ell(z)} \subset \Om_1 \}\,,\quad \mathcal I_\ell^2= \{z\in L_\ell~ :~\overline{Q_\ell(z)} \subset \Om_2\}\quad
{\rm and}~\mathcal I_\ell = \mathcal I_\ell^1 \cup \mathcal I_\ell^2\,,\end{equation}
where $Q_\ell(z)$ denotes the square of center $z$ and side-length $\ell$.
By Assumption~\ref{assump}, the number
$$N={\rm Card} \,\mathcal I_\ell$$
satisfies
\begin{equation}\label{order_N}
|\Om|\ell^{-2}-O(\ell^{-1})\leq N\leq  |\Om|\ell^{-2}\quad(\ell\to0_+)\,.
\end{equation}

{\bf Step~2.} (Defining a trial state.)

For all $z \in \mathcal I_\ell$, let  $\varphi_z \in
C^2(\overline{Q_\ell(z)})$ be the function introduced in
Lemma~\ref{A_2} and
\begin{equation}\label{eq:b-R}
b_z=\frac{H}{\kp}|B_0(z)|\quad, \quad R_z=\ell\sqrt{\kp H|B_0(z)|}\,,\quad \sigma_z=\frac{B_0(z)}{|B_0(z)|}\,.
\end{equation}
The function $\varphi_z$ satisfies
\begin{equation}\label{eq:gauge*}
{\bf F}(x)=\nb \varphi_z(x)+B_0(z){\bf A}_0(x-z),\qquad (x \in Q_\ell(z))\,.
\end{equation}
We define the function $v\in H^1_0(\Om)$ as follows,
\begin{equation*}
v(x)=
\left\{
\begin{array}{ll}
e^{i\kp H \varphi_z}\ u_{b_z,R_z,\sigma_z}\big (\frac{R_z}{\ell}(x-z)\big)            \qquad & {\rm if}\ x \in Q_\ell(z)\,,\\
0 \qquad &{\rm if}~ x \in \Om \setminus \Om_\ell\,,
\end{array}
\right.
\end{equation*}
where
\begin{equation}\label{eq:Omegal}
\Om_\ell={\rm int}\left(\displaystyle\bigcup _{z \in \mathcal I_\ell} \overline{Q_\ell(z) }\right)\,,
\end{equation}
and  $u_{b_z,R_z,\sigma_z} \in H_0^1(Q_R)$ is a minimizer of the
functional in \eqref{eq:Gb} (with
$(b,R,\sigma)=(b_z,R_z,\sigma_z))$. In the sequel, we will omit the
reference to $(b_z,R_z,\sigma_z)$ in the notation
$u_{b_z,R_z,\sigma_z}$ and write simply
\begin{equation}\label{eq:uz}
u_z=u_{b_z,R_z,\sigma_z}.
\end{equation}

{\bf Step 3.} (Energy of the trial state).

We compute the energy of the configuration $(v,{\bf F})$. We have
the obvious identities (cf. \eqref{eq:local0} and \eqref{eq:F})
\begin{equation}\label{eq:ub-decomp}
\begin{aligned}
\mathcal E(v,{\bf F};\Om) &=\int_{\Om}\Big( |(\nb-i \kp H {\bf F})v|^2-\kp^2|v|^2 +\frac 12 \kp^2|v|^4 \Big)\,dx\\
&=\displaystyle \sum_{z \in \mathcal I_\ell} \mathcal E_0(v, {\bf F};Q_\ell(z)).
\end{aligned}
\end{equation}
Using \eqref{eq:gauge*}, we write
$$
\mathcal E_0(v,{\bf F};Q_\ell(z))
=\mathcal E_0\big(e^{-i\kp H \varphi_z}v,\sigma_z|B_0(z)|{\bf A}_0(x-z);Q_\ell(z)\big)\,.
$$


By doing the  change of variable $y=\displaystyle\frac
{R_z}{\ell}(x-z)$, we get
$$
\mathcal E_0(e^{-i\kp H \varphi_z}v,\sigma_z|B_0(z)|{\bf
A}_0(x-z);Q_\ell(z))
%
= \frac{\kp}{H|B_0(z)|}\int_{Q_{R_z}}\Big(b_z\big|\big(\nb_y -i
\sigma_z{\bf A}_0\big)u_z\big|^2-|u_z|^2 +\frac 1 2|u_z|^4 \Big)\,
dy\,
$$
where $u_z$ is the function in \eqref{eq:uz} and
$(b_z,R_z,\sigma_z)$ is introduced in \eqref{eq:b-R}. By using
\eqref{eq:inf}, we get
$$\mathcal E_0(v,{\bf F};Q_\ell(z))=\mathcal E_0\big(e^{-i\kp H \varphi_z}v,\sigma_z|B_0(z)|{\bf A}_0(x-z);Q_\ell(z)\big)=\frac 1 {b_z} m_0(b_z,R_z)\,.$$
Since $\ell=\kappa^{-1/2}$ and $H\geq c_1\kappa$,  $R_z\geq 1$  (cf.
\eqref{eq:b-R}). We use  Theorem~\eqref{thm:g_prop} to write
$$m_0(b_z,R_z)\leq g(b_z)R_z^2+CR_z\,.$$
Consequently,
\begin{equation}\label{eq:E0_g}
\mathcal E_0(v,{\bf F};Q_\ell(z))\leq \ell^2\kp^2 g\Big(\frac{H}{\kp}|B_0(z)|\Big)+C \ell \kp  .
\end{equation}
We insert \eqref{eq:E0_g} into \eqref{eq:ub-decomp} to get
\begin{align*}
\mathcal E(v,{\bf F};\Om)&=\displaystyle \sum_{z \in \mathcal I_\ell} \Big(\ell^2\kp^2 g\Big(\frac{H}{\kp}|B_0(z)|\Big)+C \ell \kp\Big)\\
&\leq \kp^2 \int_{\Om_\ell} g\Big(\frac{H}{\kp}|B_0(x)|\Big)\,dx+C \ell \kp N\,,
\end{align*}
where $N={\rm Card} \,\mathcal I_\ell$.
Now, using~\eqref{order_N} and the fact that $-\frac12\leq
g(\cdot)\leq 0$, we get
$$\mathcal E(v,{\bf F};\Om)
\leq  \kp^2 \int_{\Om} g\Big(\frac{H}{\kp}|B_0(x)|\Big)\,dx+\frac12|\Omega\setminus\Omega_\ell|\kappa^2+C \frac \kp  \ell\,.
$$
To finish the proof of Proposition~\ref{prop:up}, we use that
$\Es\leq \mathcal E(v,{\bf F};\Om)$, $\ell=\kappa^{-1/2}$ and that
the regularity  of $\partial \Om$ together with
Assumption~\ref{assump} yields 
\begin{equation}\label{eq:Om_size}
|\Om \setminus \Om_\ell|=O(\ell) \,~ {\rm as} \, \ell\to0_+.
\end{equation}
\end{proof}
\section{A Priori Estimates}\label{sec:a-p-est}
In the derivation of a lower bound of the energy in \eqref{eq:GL}
various error terms arise. These terms are controlled by the
estimates that we derive in this section.

In Proposition~\ref{prop:psi}, we state a celebrated estimate of the
order parameter:
\begin{proposition}\label{prop:psi}
If $(\psi,{\bf A}) \in H^1(\Om,\C)\times H^1(\Om,\R^2)$ is a weak solution to~\eqref{eq:Euler}, then
$$\|\psi\|_{L^\infty}\leq 1.$$
\end{proposition}
A detailed proof of Proposition~\ref{prop:psi} can be found
in~\cite[Proposition~10.3.1]{Fr_Hl_10}. It only needs the assumption 
that $B_0\in L^2(\Omega)$.

Proposition~\ref{prop:psi} is used in the next theorem to give a
priori estimates on the solutions  of the Ginzburg-Landau
equations~\eqref{eq:Euler}.\\

\begin{thm}\label{thm:priori}
Let $0<c_1<c_2$ and $\alpha\in(0,1)$ be constants. Suppose that the
conditions in Assumption~\ref{assump} hold.

There exist two constants $\kappa_0>0$ and $C>0$ such that, if 
\eqref{eq:order} holds  and $(\psi,{\bf A})\in H^1(\Omega)\times H^1_{\rm div}(\Omega)$ is a solution of 
\eqref{eq:Euler}, then 
\begin{enumerate}
\item $\|(\nb-i \kp H{\bf A})\psi\|_{L^2(\Om)}\leq C\kp$\,;
\item $\displaystyle\|{\rm curl(}{\bf A-F)}\|_{L^2(\Om)}\leq \frac C
\kp$\,;
\item $\Ab-\Fb\in H^2(\Omega)$ and $\displaystyle\|\Ab-\Fb\|_{H^2(\Om)}\leq \frac C \kp$\,;
\item $\Ab-\Fb\in C^{0,\alpha}(\overline{\Om})$ and $\displaystyle\|\Ab-\Fb\|_{C^{0,\alpha}(\overline \Om)}\leq \frac C \kp$\,.
\end{enumerate}
\end{thm}

\begin{proof}
%
The
inequalities in items (1) and (2) in Theorem~\ref{thm:priori} follow from Lemma~10.3.2
in \cite{Fr_Hl_10}.

Now we prove item (3) in Theorem~\ref{thm:priori}. Let $a={\bf
A}-{\bf F} \in \Hd$. By \eqref{eq:Euler}, we know that $a\in
H^1(\Omega)$ and $\curl a\in H^1_0(\Omega)$. Using
Lemma~\ref{lem:reg} and the second equation in \eqref{eq:Euler}, we get $a\in H^2(\Omega)$ and,
$$\|{\bf A}-{\bf F}\|_{H^2(\Omega)}\leq C\|\nabla(\curl({\bf A}-{\bf
F}))\|_{L^2(\Omega)}= \frac{C}{\kappa
H}\|\overline{\psi}\,(\nabla-i\kappa H{\bf
A})\psi\|_{L^2(\Omega)}\,.$$ Using the bound $H\geq c_1\kappa$, Proposition~\ref{prop:psi} and the estimate in item (2) in Theorem~\ref{thm:priori},
we get the estimate in item (3) above.

Finally, the conclusion in item (4) in Theorem~\ref{thm:priori} is
simple a consequence of the conclusion in item (3) and  the
Sobolev embedding of $H^2(\Om)$ in $C^{0,\alpha}(\overline \Om)$.
\end{proof}

\begin{rem}
In Theorem~\ref{thm:priori}, the constant $C$ depends on $\alpha$
only in item (4).  Later in this paper, a fixed value of $\alpha$ is
chosen. For this reason, we simply denote this constant by $C$
instead of $C(\alpha)$.
\end{rem}

\begin{rem}
In Theorem~\ref{thm:priori}, Assumption~\ref{assump} is used in the derivation of items (3) and (4) only. In fact, Assumption~\ref{assump} ensures that the domains $\Omega_1$ and $\Omega_2$ satisfy the {\it cone condition}, which in turn allows us to use the Sobolev embedding theorems (cf. e.g. the proof of Lemma~\ref{lem:reg}).
\end{rem}

\section{Energy Lower Bound}\label{sec:lb}

The aim of this section is to establish a lower bound for the ground state energy in \eqref{eq:gr_st}. This will be done in two steps (cf. Lemma~\ref{lem:E0_lower} and Proposition~\ref{prop:E0_g_up} below). As a consequence of the results in this section, we will be able  to finish the proof of Theorem~\ref{thm:En_estim} and Corollary~\ref{cor:psi_estim}.

Recall that $Q_\ell(x_0)$ denotes the square of center $x_0$ and side length $\ell$. In the statements of Lemma~\ref{lem:E0_lower}, Proposition~\ref{prop:E0_g_up} and Theorem~\ref{thm:lower_bound}, we will use the functional $\mathcal E_0$ in \eqref{eq:local0}. 

\begin{lemma}\label{lem:E0_lower}
Let $\alpha \in (0,1)$ and $0<c_1<c_2$ be constants. There exist positive constants $C$ and $\kp_0$ such that, if 
\begin{itemize}
\item \eqref{eq:order} holds\,;
\item $0<\delta<1$, $0<\ell<1$, $x_0\in\Omega$\,;
\item $\overline{Q_\ell(x_0)}\subset \Omega_i$ for some $i\in\{1,2\}$\,;
\item $(\psi,{\bf A}) \in H^1(\Om,\C)\times\Hd$ is  a critical point of~\eqref{eq:GL}\,;
\item $h \in C^1(\overline{\Om}),\ \|h\|_\infty \leq 1$\,;
\end{itemize}
then  the following inequality holds 
%
%
\begin{multline*}
\mathcal E_0(h\psi,{\bf A};Q_\ell(x_0) \geq (1-\delta)\mathcal E_0(e^{-i\kp H \eta }h\psi,\sigma_\ell(x_0)|B_0(x_0)|{\bf A}_0(x-x_0);Q_\ell(x_0))\\
-C  \big(\delta^{-1}\ell^{2\alpha+2}\kappa^2+\delta \ell^2 \kappa^2 \big)\,,
\end{multline*}
for some function  $\eta \in C^2(\overline{Q_\ell(x_0)})$\,.
%
\end{lemma}
\begin{proof}
Let $\phi_{x_0}(x)=\big({\bf A}(x_0)-{\bf F}(x_0) \big)\cdot x$.
Using Theorem~\ref{thm:priori}, we get for all $x \in Q_\ell(x_0)$ 
\begin{equation}\label{eq:F_phi}
\begin{aligned}
|{\bf A}(x)-\nb\phi_{x_0}(x)-{\bf F}(x)|& \leq \|{\bf (A-F)}\|_{C^{0,\alpha}(\overline \Om_i)} |x-x_0|^\alpha\\
&\leq \frac C \kp \ell^\alpha.
\end{aligned}
\end{equation}
Define $\eta=\varphi_{x_0}+\phi_{x_0}$ where $\varphi_{x_0}$ is the function introduced in Lemma~\ref{A_2} and satisfying
        $${\bf F}(x)=\nb \varphi_{x_0}(x)+\sigma_{x_0}|B_0(x_0)|{\bf A}_0(x-x_0),\quad
\sigma_{x_0}=\frac{B_0(x_0)}{|B_0(x_0)|}\,,\quad x \in Q_\ell(x_0)\,. $$
Let
\begin{equation}\label{eq:u}
u=e^{-i\kp H \eta}h\psi.
\end{equation}
Using the gauge invariance, the Cauchy-Schwartz inequality and \eqref{eq:F_phi}, we write
$$
\big|(\nb-i\kp H{\bf A})h\psi \big|^2
                                        \geq   (1-\delta) \big|\big(\nb-i \kp H (\sigma_{x_0}|B_0(x_0)| {\bf A}_0(x-x_0))\big)u\big|^2-C \delta^{-1}\ell^ {2\alpha} \kp^2 h^2|\psi|^2 .
$$
Now,  by recalling the definition of $u$ and by using the estimates $\|\psi\|_\infty\leq1$ and $\|h\|_\infty\leq1$, we deduce the following lower bound of $\mathcal E_0(h\psi,{\bf A};Q_\ell(x_0))$,

$$\mathcal E_0(h\psi,{\bf A};Q_\ell(x_0))
                                     \geq (1-\delta)\mathcal  E_0\big(u,\sigma_{x_0}|B_0(x_0)| {\bf A}_0(x-x_0);Q_\ell(x_0)\big)-C\big(\delta^{-1}\ell^{2\alpha+2}\kp^2+\delta \ell^2 \kp^2\big).
$$
\end{proof}

\begin{proposition}\label{prop:E0_g_up}
Under the assumptions in Lemma~\ref{lem:E0_lower}, it holds
\begin{equation}\label{eq:E0_g_up}
\mathcal E_0(h\psi,{\bf A};Q_\ell(x_0))\geq \ell^2 \kp^2  g\Big(\frac H \kp |B_0(x_0)|\Big)- C \big(\ell^{2\alpha+1}\kp^2+\ell \kp+\ell^3 \kp^2 \big)\,.
\end{equation}
\end{proposition}
\begin{proof}
 Let 
$$b=\frac {H}{\kp} |B_0(x_0)|, \quad R=\ell\sqrt{\kp H |B_0(x_0)|}\,,\quad \sigma_{x_0}=\frac{B_0(x_0)}{|B_0(x_0)|}\,,$$
and define the rescaled function $v(x)=u(\frac \ell R x+x_0)$ for all $x \in Q_R=(-R/2,R/2)^2$, where the function $u$ is defined in~\eqref{eq:u}.
The change of variable $y=\frac {R}{\ell}(x-x_0)$ yields
\begin{equation}\label{eq:E1}
\mathcal  E_0\big(u,\sigma_{x_0}|B_0(x_0)| {\bf A}_0(x-x_0);Q_\ell(x_0)\big)=\frac 1b G_{b,Q_R}^{\sigma_{x_0}}(v)\,.
\end{equation}
Since $v \in H^1(Q_R)$, then by \eqref{m0-mN}, \eqref{eq:m} and Theorem~\ref{thm:g_prop},
$$
\begin{aligned}
\mathcal  E_0\big(u,\sigma_{x_0}|B_0(x_0)| {\bf A}_0(x-x_0);Q_\ell(x_0)\big)&\geq \frac 1b m(b,R)\\
&\geq \frac1b\Big(g(b)R^2-CR\Big)\,.
\end{aligned}
$$
Inserting this into the estimate in Lemma~\ref{lem:E0_lower}
{ and taking $\delta=\ell$}, we finish the proof of 
Proposition~\ref{prop:E0_g_up}.
\end{proof}

In the next theorem, we establish a lower bound of the local energy in an open subset $D$ of $\Om$. Note that,  in the particular case $h=1$ and $D=\Omega$, Theorem~\ref{thm:lower_bound} yields the lower bound in Theorem~\ref{thm:En_estim}.

\begin{theorem}\label{thm:lower_bound}
Let $\tau \in (\frac32,2)$ and $0<c_1<c_2$ be constants. Suppose that $D\subset\Omega$ is an open set with a smooth boundary. There exist constants $C>0$ and $\kappa_0>0$ such that, if
\begin{itemize}
\item \eqref{eq:order} holds\,;
\item $(\psi,{\bf A}) \in H^1(\Om,\C)\times\Hd$ is  a critical point of~\eqref{eq:GL}\,;
\item $h \in C^1(\overline{\Om}),\ \|h\|_\infty \leq 1$\,;
\end{itemize}
then  the following inequality holds 
$$\mathcal E_0(h\psi,{\bf A};D)\geq \kp^2 \int_D g\Big(\frac H \kp |B_0(x_0)|\Big)\,dx- C \kp^\tau\,.$$ 
\end{theorem}

\begin{proof}
Let $\ell\in(0,1)$ and $\alpha\in(0,1)$ be defined as follows
\begin{equation}\label{eq:tau}
\alpha=\frac 1 {2(\tau-1)}\qquad{\rm and}\qquad \ell=\kp^{1-\tau}
\end{equation}
In particular, we observe that  $\ell$ is a function of $\kappa$ such that $\ell\ll 1$ as $\kappa\to\infty$. 
Consider the lattice $L_\ell=\ell\mathbb Z\times\ell\mathbb Z$ as  in Proposition~\ref{prop:up}. Let
\begin{equation}\label{eq:I-ell-D}
\begin{aligned}
&\mathcal I_\ell^1(D)= \{z \in L_\ell~:~\overline{Q_\ell(z)} \subset D\cap\Om_1 \}\\
&\mathcal I_\ell^2(D)= \{z\in L_\ell~ :~\overline{Q_\ell(z)} \subset D\cap\Om_2\}\\
&\mathcal I_\ell(D) = \mathcal I_\ell^1 (D)\cup \mathcal I_\ell^2(D)\,,
\end{aligned}
\end{equation}
$$N_1={\rm Card} \Big(\mathcal I_\ell^1(D)\Big)\,,\quad N_2={\rm Card} \Big(\mathcal I_\ell^2(D)\Big)\,,\quad   N=N_1+N_2={\rm Card} \Big(\mathcal I_\ell(D)\Big)\,,$$
and
$$
D_\ell={\rm int}\left(\displaystyle\bigcup _{z \in \mathcal I_\ell(D)} \overline{Q_\ell(z) }\right)\,.
$$
Notice that 
\begin{equation}\label{eq:ND}
|D|\ell^{-2}-\mathcal O(\ell^{-1})\leq N\leq  |D|\ell^{-2}\,,
\end{equation}
and
\begin{equation}\label{eq:D-Dl}
|D \setminus D_\ell|=O(\ell)\quad(\ell\to0_+)\,.
\end{equation}
Recall $b_z$ and $R_z$ defined in~\eqref{eq:b-R},
$$b_z=\frac{H}{\kp}|B_0(z)|\quad{\rm and} \quad R_z=\ell\sqrt{\kp H|B_0(z)|}\,,$$
Let $(\psi,{\bf A})$ be a minimizer of~\eqref{eq:GL}. We decompose $\mathcal E_0(h\psi,{\bf A};D)$ as follows~:
\begin{equation}\label{eq:lb-decomp*}
\mathcal E_0(h\psi,{\bf A};D)=\mathcal E_0(h\psi,{\bf A};D_\ell)+\mathcal E_0(h\psi,{\bf A};D \setminus D_{\ell})\,.\end{equation}
Using \eqref{eq:D-Dl}, $|\psi| \leq 1$, $|h|\leq 1$ and Item~(1) in Theorem~\ref{thm:priori}, we get
\begin{align}
|\mathcal E_0(h\psi,{\bf A};D \setminus D_\ell)| \leq  & \int_{D \setminus D_\ell} \big(| (\nb-i\kp H{\bf A})h\psi|^2+\kp^2 h^2|\psi|^2+\frac {\kp^2}{2}h^4|\psi|^4 \big)\,dx   \nonumber \\
                                                   \leq & C \kp^2|D \setminus D_\ell|    \nonumber \\
                                                                                                     \leq & C \kp^2 \ell. \label{eq:Dl1}
\end{align}
On the other hand, we have the obvious decomposition of the energy in $D_\ell$ as follows
$$\mathcal E_0(h\psi,{\bf A};D_\ell) = \displaystyle \sum_{z \in \mathcal I_\ell(D)} \mathcal E_0(h\psi,{\bf A};Q_\ell(z))\,.$$
Using Proposition~\ref{prop:E0_g_up} and the estimate in \eqref{eq:ND}, we get
\begin{equation*}
\mathcal E_0(h\psi,{\bf A};D_\ell)\geq \ell^2  \kp^2 \displaystyle \sum_{z \in \mathcal I_\ell(D)} g\Big(\frac H \kp |B_0(z)|\Big) - C \big(\ell^{2\alpha-1}\kp^2+\ell^{-1} \kp+\ell \kp^2 \big)\,.
\end{equation*}
             %
Since $g(\cdot)\leq 0$, $D_\ell\subset D$ and $B_0$ is a step function,
$$\ell^2   \displaystyle \sum_{z \in \mathcal I_\ell(D)} g\Big(\frac H \kp |B_0(z)|\Big)
=\int_{D_\ell}g\Big(\frac H \kp |B_0(x)|\Big)\,dx\geq \int_{D}g\Big(\frac H \kp |B_0(x)|\Big)\,dx\,. $$
Consequently, we get the following lower bound
\begin{equation}
\mathcal E_0(h\psi,{\bf A};D_\ell)
                                  \geq  \kp^2 \int_{D} g\Big(\frac H \kp |B_0(x)|\Big)\,dx - C \big(\ell^{2\alpha-1}\kp^2+\ell^{-1} \kp+\ell \kp^2 \big)\,. \label{eq:Dl2}
\end{equation}
Now, the choice of $\ell$ and $\alpha$ in \eqref{eq:tau}  allows us to infer from \eqref{eq:Dl1} and \eqref{eq:Dl2}
$$
\mathcal E_0(h\psi,{\bf A};D \setminus D_\ell)\geq -C \kp^\tau\quad{\rm and}\quad
\mathcal E_0(h\psi,{\bf A};D_\ell)\geq \kp^2 \int_{D} g\Big(\frac H \kp |B_0(x)|\Big)\,dx -C \kp^\tau\,.
$$
Inserting these two estimates in \eqref{eq:lb-decomp*} finishes the proof of Theorem~\ref{thm:lower_bound}.
\end{proof}

\section{Proof of the main theorems: Energy and $L^4$-norm asymptotics}\label{sec:proof}

This section is devoted to the proof of Theorem~\ref{thm:En_estim}, Corollary~\ref{cor:psi_estim} and Theorem~\ref{thm:psi-estim}.

\begin{proof}[Proof of Theorem~\ref{thm:En_estim}]
The proof follows by collecting the results in Proposition~\ref{prop:up} and Theorem~\ref{thm:lower_bound}.
\end{proof}

\begin{proof}[Proof of Corollary~\ref{cor:psi_estim}]

Let $(\psi,\Ab)$ be a critical point of \eqref{eq:GL}. In light of \eqref{eq:Euler}, the function $\psi$ satisfies
\begin{equation}\label{IBP}
-\big(\kn\big)^2\psi=\kp^2(1-|\psi|^2)\psi\ {\rm in}\ \Om\,.
\end{equation}
Multiply both sides of~\eqref{IBP} by $\overline{\psi}$, integrate by parts over $\Om$ and use the boundary condition in \eqref{eq:Euler} to obtain
\begin{equation}\label{psi_en}
\mathcal E_0(\psi,{\bf A},\Om)=-\frac 12 \kappa^2 \int_\Om |\psi|^4\,dx\,.
\end{equation}
Here $\mathcal E_0$ is the function in \eqref{eq:local0}. Using Theorem~\ref{thm:lower_bound} with $h=1$ and $D=\Omega$, we get
\begin{equation}\label{up_psi}
\int_\Om |\psi|^4\leq -2 \int_\Om g\left(\frac H \kappa |B_0(x)|\right)\, dx +C \kappa^{\tau-2}\,.
\end{equation}
Now, we assume in addition that $(\psi,{\bf A})$ is a minimizer of~\eqref{eq:GL}.  We will determine a lower bound of $\|\psi\|_4$ matching with the upper bound in \eqref{up_psi}. To that end, we observe that
\begin{equation}\label{eq:corol-psi}
\kappa^2H^2\int_\Omega|\curl\Ab-B_0|^2\,dx+\mathcal E_0(\psi,\Ab;\Omega)=\Es\,.\end{equation}
This yields that
$$\mathcal E_0(\psi,\Ab;\Omega)\leq \Es\,.$$
The upper bound in Theorem~\ref{prop:up} gives
$$\mathcal E_0(\psi,{\bf A};\Om)
                               \leq \kappa^2 \int_\Om g\left(\frac H \kappa |B_0(x)|\right)\, dx +C \kappa^{3/2}\,.
$$
Inserting this into \eqref{psi_en}, we obtain
\begin{equation}\label{low_psi}
\int_\Om |\psi|^4\geq -2 \int_\Om g\left(\frac H \kappa |B_0(x)|\right)\, dx -C \kappa^{1/2}\,.
\end{equation}
It remains to prove the estimate of the magnetic energy. In fact, we get by \eqref{eq:corol-psi}, Proposition~\ref{prop:up} and Theorem~\ref{thm:lower_bound}
$$\kappa^2H^2\int_\Omega|\curl\Ab-B_0|^2\,dx\leq \Es-\mathcal E_0(\psi,\Ab;\Omega)\leq C\kappa^\tau\,.$$
\end{proof}

\begin{proof}[Proof of Theorem~\ref{thm:psi-estim}]

The proof we give does not use {\it a priori} elliptic estimates, hence, it diverges from the one used by Sandier-Serfaty in \cite{SS02}. As a substitute of the elliptic estimates, we use   the {\it a priori} estimates obtained by an energy argument in Theorem~\ref{thm:priori}, and a trick that we borrow from a paper by Helffer-Kachmar in \cite{HeKa}.

Note that the particular case $D=\Om$ is handled by Corollary~\ref{cor:psi_estim}.
Now we establish Theorem~\ref{thm:psi-estim} in the general case when $D$ is an open subset of $\Om$ with a smooth boundary. The proof  is decomposed into two steps:

{\bf Upper bound:}

For $\ell \in (0,1)$, we define the two sets 
\begin{equation}\label{eq:Dl'*}
\begin{aligned}
&D_\ell=\{x \in D ~: ~{\rm dist}(x,\partial D)\geq \ell\}\,,\\
&\big(\Omega\setminus D\big)_\ell=\{x \in \Omega\setminus D ~: ~{\rm dist}(x,\partial D)\geq \ell\}\,.
\end{aligned}
\end{equation}
Since the boundary of $D$ is smooth, we get 
\begin{equation}\label{eq:Dl'}
|D\setminus D_\ell|=\mathcal O(\ell)\quad(\ell\to0_+)\,.\end{equation}
However, the boundary of $\Omega\setminus D$ might have a finite number of cusps. But we still have,
\begin{equation}\label{eq:Dl'**}
\epsilon(\ell):=\Big|\big(\Omega\setminus D\big)\setminus \big(\Omega\setminus D\big)_\ell\Big|\to0\quad{\rm as~}\ell\to0_+\,.
\end{equation}
Let $\chi_\ell \in C_c^\infty(\R^2)$ be a cut-off function satisfying
\begin{equation}\label{eq:chi}
0\leq \chi_\ell\leq 1\ {\rm in}\ \R^2, \quad {\rm supp}\chi_\ell \subset D,\quad \chi_\ell=1\ {\rm in}\ D_\ell \quad {\rm and} \quad |\nb \chi_\ell|\leq \frac C\ell\ {\rm in}\ \R^2\,,
\end{equation}
where $C$ is a universal constant.\\
Multiply both sides of the equation
$$-\big(\kn\big)^2\psi=\kp^2(1-|\psi|^2)\psi$$
by $\chi_\ell^2 \overline{\psi}$,  integrate by parts over $D$, then  use~\eqref{eq:chi}, \eqref{eq:Dl'} and $\|\psi\|_\infty \leq1$ to get
\begin{align}
\int_D \Big( \big|(\nb-i\kp H {\bf
A})\chi_\ell\psi\big|^2-\kp^2 \chi_\ell^2|\psi|^2+\kp^2 \chi_\ell^2|\psi|^4 \Big)\,dx &= \int_D |\nb \chi_\ell|^2|\psi|^2\,dx \nonumber\\
                                                                                      &\leq C\ell^{-1}\,. \label{l1}
\end{align}
Notice that $\chi_\ell^4\leq \chi_\ell^2 \leq 1$. We infer from~\eqref{l1}  
\begin{equation}\label{eq:**}
\mathcal E_0(\chi_\ell\psi,{\bf A};D) \leq -\frac 12 \kp^2 \int_D\chi_\ell^4 |\psi|^4\,dx+C \ell^{-1}\,.\end{equation}
Again, using ~\eqref{eq:chi}, \eqref{eq:Dl'} and $\|\psi\|_\infty \leq1$,  we obtain
\begin{align*}
\int_D |\psi|^4 \,dx &= \int_D \chi_\ell^4 |\psi|^4\,dx +\int_D(1- \chi_\ell^4) |\psi|^4\,dx \\
                       &\leq \int_D \chi_\ell^4 |\psi|^4\,dx+C\ell \,.
\end{align*}
Consequently, in light of \eqref{eq:**}, we write,
$$\mathcal E_0(\chi_\ell\psi,{\bf A};D) \leq -\frac 12 \kp^2 \int_D|\psi|^4\,dx+C \left(\ell \kp^2+\ell^{-1}\right)\,.$$
Now Theorem \ref{thm:lower_bound} with $h=\chi_\ell$ yields
$$ \int_D|\psi|^4\,dx \leq -2 \int_D g\Big(\frac H \kp |B_0(x_0)|\Big)\, dx+ C \left(\ell+\ell^{-1}\kp^{-2} +\kp^{\tau-2}\right)\,.$$
Choosing $\ell=\kp^{-\frac 12}$, we get
\begin{equation}\label{eq:up_psi}
\int_D|\psi|^4\,dx \leq -2 \int_D g\Big(\frac H \kp |B_0(x_0)|\Big)\, dx+ C \kp^{\tau-2}\,. 
\end{equation}
%
In a similar fashion, by using the alternative property in \eqref{eq:Dl'**}, we get,
\begin{equation}\label{eq:up_psi*}
\begin{aligned}
\int_{\Omega\setminus D}|\psi|^4\,dx &\leq -2 \int_{\Omega\setminus D} g\Big(\frac H \kp |B_0(x_0)|\Big)\, dx+ C \kp^{\tau-2}+C\epsilon(\ell)\\
&=-2 \int_{\Omega\setminus D} g\Big(\frac H \kp |B_0(x_0)|\Big)\, dx+o(1)\,. 
\end{aligned}
\end{equation}

{\bf Lower bound:}

Since $\partial D$ is smooth, $|\partial D|=0$ and the following simple decomposition holds
$$\int_D |\psi|^4 \,dx =\int_\Om |\psi|^4 \,dx -\int_{\Omega\setminus D} |\psi|^4 \,dx\,.$$
By~\eqref{eq:up_psi*}, we have
$$\int_{\Omega\setminus D}|\psi|^4\,dx \leq -2 \int_{\Omega\setminus D} g\Big(\frac H \kp |B_0(x_0)|\Big)\, dx+ o(1)\,.$$
Using~\eqref{low_psi} and \eqref{eq:up_psi}, we deduce that
$$\displaystyle\int_D |\psi|^4 \,dx=\int_D g\Big(\frac H \kp |B_0(x_0)|\Big)\, dx+o(1)\,.$$
\end{proof}

\section{Exponential decay and proof of Theorem~\ref{thm:exp-dec}}\label{sec:exp-dec}

A key ingredient in the proof of Theorem~\ref{thm:exp-dec} is

\begin{lem}\label{lem:exp-dec}
Let $\alpha\in(0,1/4)$ be a constant. Under the assumptions in Theorem~\ref{thm:exp-dec}, there exist two constants $\kappa_0>0$ and $C>0$ such that
for all $\kappa\geq \kappa_0$ and $\phi\in C_c^\infty(\Omega_1)$,
$$\|(\nabla-i\kappa H\Ab)\phi\|^2_{L^2(\Omega_1)}\geq {\kp H}\Big(1-C\kappa^{-\alpha}\Big)\|\phi\|^2_{L^2(\Omega_1)}\,.$$ 
\end{lem}
\begin{proof}
Let $\phi\in C_c^\infty(\Omega_1)$. We write the following well known spectral estimate (cf. e.g. \cite[Lem.~1.4.1]{Fr_Hl_10})
$$\int_{\Omega_1}|(\nabla-i\kappa H\Ab)\phi|^2\,dx\geq \kappa H\int_{\Omega_1}\curl\Ab|\phi|^2\,dx\,.$$
Using the simple decomposition $\curl\Ab=(\curl \Ab-B_0)+B_0$, $B_0=1$ in $\Omega_1$, the Cauchy-Schwartz inequality and the estimate of $\|\curl\Ab-B_0\|_2$ in Corollary~\ref{cor:psi_estim} (with $\tau=2-2\alpha$), we get,
 $$
\begin{aligned}
\int_{\Omega_1}|(\nabla-i\kappa H\Ab)\phi|^2\,dx&\geq\kappa H\int_{\Omega_1}|\phi|^2\,dx+\kappa H\int_{\Omega_1}\big(\curl\Ab-B_0\big)|\phi|^2\,dx\\
&\geq\kappa H\int_{\Omega_1}|\phi|^2\,dx-\kappa H\|\curl\Ab-B_0\big\|_{L^2(\Omega)}
\|\phi\|_{L^4(\Omega_1)}^2\\
&\geq  \kappa H\Big[\|\phi\|_{L^2(\Omega_1)}^2-C\kappa^{-1-\alpha}\|\phi\|_{L^4(\Omega_1)}^2\Big]\,.
\end{aligned}
$$
Now, we use the Sobolev embedding of $H^1(\Omega_1)$ into $L^4(\Omega_1)$ and the diamagnetic inequality to get, for all $\eta\in(0,1)$ (cf. \cite[Eq.~(4.14)]{FK-jde}),
$$
\begin{aligned}
\|\phi\|_{L^4(\Omega_1)}^2&\leq C_{\rm Sob}\Big(\eta\|\,\nabla|\phi|\,\|_{L^2(\Omega_1)}^2+\eta^{-1}\|\phi\|_{L^2(\Omega_1)}^2\Big)\\
& \leq C_{\rm Sob}\Big(\eta\|(\nabla-i\kappa H\Ab)\phi\|_{L^2(\Omega_1)}^2+\eta^{-1}\|\phi\|_{L^2(\Omega_1)}^2\Big)\,.
\end{aligned}
$$
To finish the proof, we select $\eta=\kappa^{-1}$.
\end{proof}

Now, the proof of Theorem~\ref{thm:exp-dec} follows similarly as \cite[Thm.~4.1]{FK-jde}.

\begin{proof}[Proof of Theorem~\ref{thm:exp-dec}]
Define the distance function $t$ on $\Om_1$ as follows :
$$t(x)={\rm dist}(x,\partial\Om_1)\,.$$
Let $\tilde{\chi} \in C^\infty(\R)$ be a function satisfying
$$\tilde{\chi}=0\ {\rm on}\ (-\infty,\frac12],\quad \tilde{\chi}=1\ {\rm on}\ [1,\infty)\ {\rm and}\quad  |\nb \tilde{\chi}|\leq m\,,$$
where $m$ is a universal constant.\\
Define the two  functions $\chi$ and $f$ on $\Om_1$ as follows :
$$\chi(x)=\tilde{\chi}\big(\sqrt{\kp H} t(x)\big)\,,$$
$$f(x)=\chi(x) \exp \big(\varepsilon \sqrt{\kp H}\,t(x)\big)\,,$$
where $\varepsilon$ is a positive number whose value will be fixed later.\\
Using the first equation of~\eqref{eq:Euler}, we multiply both sides by $f^2\overline\psi$ and we integrate by parts over $\Om_1$, it results
\begin{align}
\int_{\Om_1} \Big( \big|(\nb-i\kp H {\bf
A})f\psi\big|^2-|\nb f|^2|\psi|^2 \Big)\,dx &= \kp^2\int_{\Om_1}\Big( |\psi|^2-|\psi|^4 \big) f^2\,dx  \nonumber\\
                                            & \leq \kp^2\int_{\Om_1} |\psi|^2 f^2\,dx\,. \label{exp1}                                                  
\end{align}
We combine the conclusions in~\eqref{exp1} and Lemma~\ref{lem:exp-dec} to get 
\begin{align}
\int_{\Om_1}|\nb f|^2|\psi|^2\,dx &\geq \big(\kp H(1-C \kp^{-\alpha})-\kp^2 \big)\|f \psi\|^2_{L^2(\Om_1)} \nonumber\\
                                  &\geq \big(\lambda-C(1+\lambda)\kp^{-\alpha}\big) \kp^2 \|f \psi\|^2_{L^2(\Om_1)}\,. \label{exp2}
\end{align} 
Now, we estimate the term on the right side of~\eqref{exp2} as follows
\begin{equation}\label{exp3}\int_{\Om_1}|\nb f|^2|\psi|^2\,dx \leq \varepsilon^2\kp H \|f \psi\|^2_{L^2(\Om_1)}+ C\kp H \int_{\Om_1\cap \{\sqrt{\kp H} t(x)\leq 1\}}|\psi|^2\,dx \nonumber \\ 
\end{equation}
Inserting~\eqref{exp3} into~\eqref{exp2} and dividing by $\kp^2$ yields
	$$\big(\lambda-C(1+\lambda)\kp^{-\alpha}-\varepsilon^2\big) \|f \psi\|^2_{L^2(\Om_1)} \leq  C \int_{\Om_1 \cap  \{\sqrt{\kp H} t(x)\leq 1\}}|\psi|^2\,dx\,.$$		
We choose the constant $\varepsilon$ such that $0<\varepsilon<\sqrt{\lambda}$. That way, we get for $\kappa$ sufficiently large,
\begin{equation}\label{exp4}
\int_{\Om_1\cap \{\sqrt{\kp H} t(x)\geq 1\}}|f \psi|^2\,dx \leq \widetilde{C} \int_{\Om_1\cap \{\sqrt{\kp H} t(x)\leq 1\}}|\psi|^2\,dx\,,	
\end{equation}
where $\widetilde{C}$ is a constant.
Inserting~\eqref{exp2} and~\eqref{exp4} into~\eqref{exp1} finishes the proof of Theorem~\ref{thm:exp-dec}.							
\end{proof}

We conclude this paper by Theorem~\ref{thm:exp-dec*} below, whose proof is similar to that of Theorem~\ref{thm:exp-dec}. 

\begin{thm}\label{thm:exp-dec*}{\bf[Exponential decay in $\Omega_2$]}~

Let $\lambda,c_2>0$ be   two constants such that
$\frac1{|a|}+\lambda< c_2$. 
There exist constants $C,\varepsilon,\kappa_0>0$ such that, if 
$$\kappa\geq\kappa_0\,,\quad \left(\frac1{|a|}+\lambda\right)\kappa \leq H\leq c_2\kappa\,,\quad (\psi,\Ab)_{\kappa,H}~{\rm is~a~minimizer~of~}\eqref{eq:GL}\,,$$ 
then
\begin{multline*}
\int_{\Omega_2\cap\{{\rm dist}(x,\partial\Omega_2)\geq \frac{1}{\sqrt{\kappa H}}\}}
\Big(|\psi|^2+{\frac{1}{\kp H}}|(\nabla-i\kappa H\Ab)\psi|^2\Big)\,\exp\Big(2\varepsilon\sqrt{\kappa H}\,{\rm dist}(x,\partial\Omega_2)\Big)dx\\
\leq
C \int_{\Omega_2\cap\{{\rm dist}(x,\partial\Omega_1)\leq \frac{1}{\sqrt{\kappa H}}\}}
|\psi|^2\,dx\,.\end{multline*}
\end{thm}

\subsection*{Acknowledgments}
AK is supported by a research grant from  Lebanese University.

\appendix
\section{Gauge transformation}\label{sec:gauge}
\begin{lem}\label{A_1}
Suppose that $\Omega$ satisfies the conditions in Assumption~\ref{assump}. Let $B_0 \in L^2(\Omega)$. There exists a unique vector field ${\bf
F} \in H^1_{\rm div}(\Omega)$ such that
$${\rm curl}\,{\bf F}=B_0\,.$$
Furthermore, $F$ is in $C^\infty(\Omega_1)\cup
C^\infty(\Omega_2)$ and in $H^2(\Omega_1)\cup H^2(\Omega_2)$.
\end{lem}
\begin{proof}
Let $f \in H^2(\Omega)\cap H_0^1(\Omega)$ be the unique solution of $-\Delta f=B_0$ in $\Omega$ (cf.~\cite{Fr_Hl_10}).

The vector field ${\bf F}=(\partial_{x_2}f,-\partial_{x_1}f)\in
H^1_{\rm div}(\Omega)$ and satisfies $\curl{\bf F}=B_0$.

Since $B_0$ is constant in $\Omega_i$ for $i\in\{1,2\}$, $f$ the
solution of $-\Delta f=B_0$ becomes in $C^\infty(\Omega_i)$. This
yields that $\bf F$ is in $C^\infty(\Omega_i)$, $i\in\{1,2\}$.

That $\Fb\in H^2(\Omega_i)$ follows from the regularity of the solution of the equation 
$-\Delta f=B_0$ with Dirichlet condition on $\partial\Omega$ and continuity conditions on $\Gamma=\partial\Omega_1\cap\partial\Omega_2$ (cf. \cite[Thm.~B.1]{Kach-aa}).
The uniqueness of ${\bf F}$ is a consequence of the estimate (cf.
\cite[Prop.~D.2.1]{Fr_Hl_10} or \cite[Appendix~1]{Te})
\begin{equation}\label{eq:curl-div}
\forall~a\in H^1_{\rm div}(\Omega)\,,\quad \|a\|_{H^1(\Omega)}\leq C_*\|\curl a\|_{L^2(\Omega)}\,,
\end{equation}
where $C_*>0$ is a universal constant.
\end{proof} 
\begin{lem}\label{A_2}
Let $i\in\{1,2\}$, $\ell \in(0,1)$,  $x_0 \in\Omega$ and
$Q_\ell(x_0)\subset\Omega_i$. There exists a function $\varphi_{x_0}
\in  C^2(Q_\ell(x_0))$ such that the magnetic potential ${\bf F}$
defined in Lemma~\ref{A_1} satisfies
$${\bf F}(x)=B_0(x_0)A_0(x-x_0)+\nb \varphi_{x_0}(x), \qquad (x \in Q_\ell(x_0))$$
where $B_0$ is the function defined in~\ref{assump} and ${\bf A_0}$
is the magnetic potential introduced in~\eqref{canon}.
\end{lem}
\begin{proof}
By the definition of ${\bf F}$ and ${\bf A}_0$ we have for all $x
\in Q_\ell(x_0)$,
$${\rm curl}\,{{\bf F}(x)}=B_0(x){\rm curl}{\bf A}_0(x)\quad{\rm in~}Q_\ell(x_0)\,.$$
Since  $Q_\ell(x_0)$ is simply connected and $B_0$ is constant in
$Q_\ell(x_0)$, we get the existence of the function $\varphi_{x_0}$.
\end{proof}

\section{$\curl$-$\Div$ elliptic estimate}\label{sec:reg}

\begin{lem}\label{lem:reg}
Suppose that $\Omega$ is simply connected and satisfies the conditions in Assumption~\ref{assump}. There exists a constant $C>0$ such that, if $a\in
H^1_{\rm div}(\Omega)$ and $\curl a\in H^1_0(\Omega)$, then $a\in H^2(\Omega)$ and the following inequality holds
$$\|a\|_{H^2(\Omega)}\leq C\|\nabla(\curl a)\|_{L^2(\Omega)}\,.$$
\end{lem}
\begin{proof}
Let $f\in H^1_0(\Omega)\cap H^2(\Omega)$ be the unique solution of
the Dirichlet problem $-\Delta f=\curl a$ in $\Omega$. By the
inequality in \eqref{eq:curl-div}, we get that $a=\nabla^\bot
f=(\partial_{x_2}f,-\partial_{x_1}f)$.

Since $f=0$ on $\partial\Omega$ and $\curl a\in H^1(\Omega)$, we have by the elliptic estimates
$$\|f\|_{H^2(\Omega)}\leq C\left(\|\curl
a\|_{L^2(\Omega)}+\|f\|_{L^2(\Omega)}\right)\,,$$ 
$f\in H^3(\Omega)$ and 
$$
\begin{aligned}
\|f\|_{H^3(\Omega)}&\leq C\left(\|\curl a\|_{H^1(\Omega)}+\|f\|_{H^2(\Omega)}\right)\\
&\leq C\left(\|\curl a\|_{H^1(\Omega)}+\|f\|_{L^2(\Omega)}\right)\,.
\end{aligned}
$$
This proves that $a\in H^2(\Omega)$. Now,
since $f=0$ and $\curl a=0$ on $\partial\Omega$, we get by the
Poincar\'e inequality
$$\|f\|_{H^3(\Omega)}\leq C\left(\|\nabla(\curl
a)\|_{L^2(\Omega)}+\|\nabla f\|_{L^2(\Omega)}\right)\,.$$ To finish the
proof of Lemma~\ref{lem:reg}, we simply observe that, since
$a=\nabla^\bot f$, $\|\nabla f\|_{L^2(\Omega)}=\|a\|_{L^2(\Omega)}$
and consequently
\begin{align*}
\|a\|_{L^2(\Omega)}&\leq C_*\|\curl a\|_{L^2(\Omega)}&[{\rm By~}\eqref{eq:curl-div}]\\
&\leq C\|\nabla (\curl a)\|_{L^2(\Omega)}&[\text{\rm By~the~Poincar\'e~inquality}]\,.
\end{align*}
\end{proof}


\begin{thebibliography}{100}

\bibitem{AHP} Y. Almog, B. Helffer, X.B. Pan. Mixed normal-superconducting states in the presence of strong electric currents. {\it arXiv:1505.06327} (2015).

\bibitem{att1} K. Attar. The ground state energy of the two dimensional Ginzburg–Landau functional with variable magnetic field. {\it Annales de l'Institut Henri Poincar\'e (C) Non Linear Analysis}. Vol. 32. No. 2. Elsevier Masson (2015).

\bibitem{att2} K. Attar. Energy and vorticity of the Ginzburg-Landau model with variable magnetic field. {\it  arXiv:1411.5479} (2014).

\bibitem{att3} K. Attar. Pinning with a variable magnetic field of the two dimensional Ginzburg-Landau model. {\it arXiv:1503.06500} (2015).

\bibitem{BF} V. Bonnaillie-Noël, S. Fournais. Superconductivity in domains with corners. {\it  Reviews in Mathematical Physics} 19.06 (2007):  607-637.

\bibitem{chapman} S.J. Chapman, Q. Du, M.D. Gunzburger. A Ginzburg–Landau type model of superconducting/normal junctions including Josephson junctions. {\it  European Journal of Applied Mathematics} 6.02 (1995): 97-114.


\bibitem{CL} A. Contreras, X. Lamy. Persistence of superconductivity in thin shells beyond Hc1. {\it  arXiv:1411.1078v1}. (2014).

\bibitem{Fr_Hl_10}
S. Fournais, B. Helffer. {\it Spectral Methods in Surface Superconductivity}. Progress in Nonlinear Differential Equations and their Applications. Vol. 77, Birkh\"auser, Boston (2010).

\bibitem{FK-cpde}
S. Fournais, A. Kachmar. The ground state energy of the three dimensional Ginzburg–Landau functional. Part I: Bulk
regime, {\it Communications in Partial Differential Equations} 38.2 (2013): 339–383.

\bibitem{FK-adm} S. Fournais, A. Kachmar. Nucleation of bulk superconductivity close to critical magnetic field. {\it Advances in Mathematics} 226.2 (2011).

\bibitem{FK-jde} S. Fournais, A. Kachmar. On the transition to the normal phase for superconductors surrounded by normal conductors.{\it  Journal of Differential Equations} 247.6 (2009): 1637-1672.

\bibitem{GP} T. Giorgi, D. Phillips. 
The breakdown of superconductivity due to strong fields for the Ginzburg-Landau model.
{\it SIAM Journal on Mathematical Analysis} 30 (1999): 341-359.

\bibitem{HeKa} B. Helffer, A. Kachmar.The Ginzburg–Landau Functional with Vanishing Magnetic Field. {\it  Archive for Rational Mechanics and Analysis} (2015).

\bibitem{HPRS} P.D. Hislop, N. Popoff, N. Raymond, M.P. Sundqvist. Band functions in presence of magnetic steps. {\it arxiv:1501.02824} (2015).

\bibitem{K-JFA} A. Kachmar. The ground state energy of the three-dimensional Ginzburg–Landau model in the mixed phase. {\it Journal of Functional Analysis} 261.11 (2011): 3328-3344.


\bibitem{Kach-aa} A. Kachmar. On the perfect superconducting solution for a generalized Ginzburg-Landau equation. {\it Asymptot. Anal.}  54 (2007): 125-164.

\bibitem{PK} X.B. Pan, K.H. Kwek. Schr\"odinger operators with non-degenerately vanishing magnetic fields in bounded domains. {\it Transactions of the American Mathematical Society} 354.10 (2002): 4201-4227.

\bibitem{Te} R. Temam. {\it  Navier-Stokes equations. Theory and numerical analysis}. Vol. 343. American Mathematical Soc. (2001).


\bibitem{SS-b} E. Sandier, S. Serfaty.  {\it Vortices in the magnetic Ginzburg-Landau model}. Progress in Nonlinear Partial Differential Equations and Their Applications. Vol. 70, Birkh\"auser (2007).

\bibitem{SS02} E. Sandier, S. Serfaty. The decrease of bulk superconductivity close to the second critical field in the Ginzburg--Landau model. {\it SIAM Journal on Mathematical Analysis} 34.4 (2003): 939-956.


\end{thebibliography}
\end{document}